\documentclass[11pt]{article}
\usepackage{amsmath}
\usepackage{amssymb,amsfonts,amsmath,amsthm}
\usepackage{epsfig}
\newtheorem{theorem}{Theorem}[section]
\newtheorem{thm}[theorem]{Theorem}

\newtheorem{lem}[theorem]{Lemma}

\newtheorem{cor}[theorem]{Corollary}

\makeatletter \@addtoreset{equation}{section}

\newcommand{\qbinom}[2]{\genfrac{[}{]}{0pt}{}{#1}{#2}}
\newcommand{\Qrk}{Q(h \Mid \br;\bk)}
\newcommand{\Mid}{\:|\:}  
\newcommand{\br}{\mathbf r}
\newcommand{\bk}{\mathbf k}
\newcommand{\Erk}{E_{\br,\bk}}
\newcommand{\lrq}[3]{\left(\frac{#1}{#2}#3\right)}

\def\CT{\mathop{\mathrm{CT}}}
\def\LC{\mathop{\mathrm{LC}}}
\begin{document}

\def\CC{\mathbb{C}}
\title{A Family of $q$-Dyson Style Constant Term Identities}
\author{
  {\vspace{0.2cm}}
  {Lun Lv$^1$, Guoce Xin$^2$, AND Yue Zhou$^3$}\\
  {\small $^{1,2,3}$Center for Combinatorics, LPMC}\\
  {\vspace{0.2cm}}
  {\small Nankai University, Tianjin 300071, P.R. China}\\
  { \small $^1$lvlun@mail.nankai.edu.cn\ \ \ $^2$gxin@nankai.edu.cn\ \ \ $^3$zhouyue@mail.nankai.edu.cn}\\
 }

\date{June 6, 2007}
\maketitle

\begin{abstract}
By generalizing Gessel-Xin's Laurent series method for proving the
Zeilberger-Bressoud $q$-Dyson Theorem, we establish a family of
$q$-Dyson style constant term identities. These identities give
explicit formulas for certain coefficients of the $q$-Dyson product,
including three conjectures of Sills' as
special cases and  generalizing Stembridge's first layer formulas for characters of
$SL(n,\mathbb{C})$.
\end{abstract}

{\small \emph{Mathematics Subject Classification}. Primary 05A30,
secondary  33D70.}

{\small \emph{Key words}. $q$-series, Dyson conjecture, Laurent
series, partial fractions, constant term}

\section{Introduction}

\subsection{Notation}
Throughout this paper, we let $n$ be a nonnegative integer, and use
the following symbols:
\begin{multline*}
\shoveright{\mathbf{a}:=(a_0,a_1,\ldots,a_n),\\}
\shoveright{a:=a_1+a_2+\cdots+a_n,\\}
\shoveright{\mathbf{x}:=(x_0,x_1,\ldots,x_n),\\}
\shoveright{(z)_n:=(1-z)(1-zq)\cdots(1-zq^{n-1}),\\}
\shoveright{D_n(\mathbf{x},\mathbf{a},q):=\prod_{0\leq i<j\leq n}
\left(\frac{x_i}{x_j}\right)_{\!\!a_i}
\left(\frac{x_j}{x_i}q\right)_{\!\!a_j},\quad\quad\quad\quad \quad \quad \quad \quad \quad \quad\quad\quad(q\mbox{-Dyson product})\\}
\shoveright{\CT_{\mathbf{x}}F(\mathbf{x})\ \mbox{means to take the
constant term in the $x$'s of the series $F(\mathbf{x})$.}}
\end{multline*}
Since our main objective in this paper is to evaluate the constant term
of the form
$$\frac{x_{j_1}^{p_1}\cdots x_{j_{\nu}}^{p_{\nu}}}{x_{i_1}x_{i_2}\cdots x_{i_m}}D_n(\mathbf{x},\mathbf{a},q),$$
it is convenient for us to define:
\begin{multline*}
\shoveright{I_{0}:=\{i_1,i_2,\ldots,i_{m}\}\ \mbox{is a set with
$0=i_1<i_2<\cdots <i_m<n$,}\\}
\shoveright{I:=I_0\setminus\{i_1\}=\{i_2,\ldots,i_m\},\\}
\shoveright{T:=\{t_1,\ldots,t_d\}\ \mbox{ is a $d$-element subset of
$I_0$ or $I$ with $t_1<t_2<\cdots<t_d$,}\\}
\shoveright{\sigma(T):=a_{t_1}+a_{t_2}+\cdots+a_{t_d},\\}
\shoveright{w_{i}:=\left\{ \begin{aligned}
         &a_i,\quad &for& \quad i\not\in T; \\
                  &0, \quad &for& \quad i\in T,
                          \end{aligned} \right.\\}
\shoveright{w:=w_1+w_2+\cdots+w_n=a-\sigma(T).}
\end{multline*}

\subsection{Main results}
In 1962, Freeman Dyson \cite{dyson1962} conjectured the following
identity:
\begin{theorem}[Dyson's Conjecture]\label{t-dyson}
For nonnegative integers $a_0,a_1,\ldots ,a_n$,
\begin{equation*}
\CT_{\mathbf{x}} \prod_{0\le i\ne j \le n}
\left(1-\frac{x_i}{x_j}\right)^{\!\!a_i} =
 \frac{(a_0+a_1+\cdots+a_n)!}{a_0!\, a_1!\, \cdots a_n!}.
\end{equation*}
\end{theorem}

Dyson's conjecture was first proved independently by
Gunson~\cite{gunson} and  by Wilson \cite{wilson}. An elegant
recursive proof was published by Good \cite{good1}.

George Andrews \cite{andrews1975} conjectured the $q$-analog of the
Dyson conjecture in 1975:
\begin{thm}\label{thm-dyson}\emph{(Zeilberger-Bressoud)}.
For nonnegative integers $a_0,a_1,\dots,a_n$,
\begin{align*}
\CT_{\mathbf{x}}\,D_n(\mathbf{x},\mathbf{a},q)=\frac{(q)_{a+a_0}}{(q)_{a_0}(q)_{a_1}\cdots(q)_{a_n}}.
\end{align*}
\end{thm}

Andrews' $q$-Dyson conjecture attracted much interest
\cite{askey1980, kadell1985, stanley-qdyson1, stanley-qdyson2,
stembridge-qdyson}, and was first proved, combinatorially, by
Zeilberger and Bressoud \cite{zeil-bres1985} in 1985. Recently,
Gessel and Xin \cite{gess-xin2006} gave a very different proof by
using properties of formal Laurent series and of polynomials. The
coefficients of the Dyson and $q$-Dyson product are researched in
\cite{breg, kadell1998,  sills2006, sills-zeilberger,
stembridge1987}. In the equal parameter case, the identity reduces
to Macdonald's constant term conjecture \cite{macdonald} for root
systems of type $A$.

The main results of this paper are the following $q$-Dyson style
constant term identities:
\begin{thm}[Main Theorem]\label{t-main-thm}
Let $i_1,\ldots,i_m$ and $j_1,\ldots,j_{\nu}$ be distinct integers
satisfying $0=i_1<i_2<\cdots <i_m<n$ and $0<j_1<\cdots <j_{\nu}\le n$.
Then
{\small\begin{align}\label{e-main}
\CT_{\mathbf{x}}\frac{x_{j_1}^{p_1}\cdots
x_{j_{\nu}}^{p_{\nu}}}{x_{i_1}x_{i_2}\cdots
x_{i_m}}D_n(\mathbf{x},\mathbf{a},q)
=\frac{(q)_{a+a_0}}{(q)_{a_0}(q)_{a_1}\cdots(q)_{a_n}}\sum_{
\varnothing\neq T\subseteq I_0}(-1)^dq^{L(T)}
\frac{1-q^{\sigma(T)}}{1-q^{1+a_0+a-\sigma(T)}},
\end{align}}
where the $p$'s are positive integers with $\sum_{i=1}^{\nu}p_i=m$ and
\begin{align}\label{e-LT}
L(T)=\sum_{l\in I_0}\sum_{i=l}^nw_i-\sum_{l=1}^{\nu}p_l\sum_{i=j_l}^nw_i.
\end{align}
\end{thm}

We remark that the cases $i_1>0$ or $i_m=n$ or both can be evaluated
using the above theorem and Lemma {\rm\ref{l-pi}}. The equal
parameter case of the above results are called by Stembridge
\cite{stembridge1987} ``the first layer formulas for characters of
$SL(n,\mathbb{C})$". The following three Corollaries are the
simplified, but equivalent, version of Sills' conjectures
{\rm\cite{sills2006}}. They are all special cases of Theorem
\ref{t-main-thm}. When $m=1$, we obtain
\begin{cor}[Conjecture 1.2, \cite{sills2006}]\label{conj-s1}
Let $r$ be a fixed integer with $0< r \leq n$ and
$n\geq 1$. Then
\begin{align}\label{e-conj1}
\CT_{\mathbf{x}}\,\frac{x_r}{x_0}\,D_n(\mathbf{x},\mathbf{a},q)=-q^{\sum_{k=1}^{r-1}a_k}
\left(\frac{1-q^{a_0}}{1-q^{a+1}}\right)\frac{(q)_{a+a_0}}{(q)_{a_0}(q)_{a_1}\cdots(q)_{a_n}}.
\end{align}
\end{cor}
When $m=2$ and $p_1=2$, we obtain
\begin{cor}[Conjecture 1.5, \cite{sills2006}]\label{conj-s2}
Let $r,t$ be fixed integers with $1\leq t< r \leq n$ and $n\geq 2$. Then
{\small\begin{align}\label{e-conj2}
\CT_{\mathbf{x}}&\,\frac{x_r^2}{x_0x_t}\,D_n(\mathbf{x},\mathbf{a},q)\nonumber\\
&{\tiny=q^{\widetilde{L}(r,t)}\left(\frac{(1-q^{a_0})(1-q^{a_t})\Big((1-q^{a_0+a+1})+q^{a_t}
(1-q^{a+1-a_t})\Big)}
{(1-q^{a+1-a_t})(1-q^{a+1})
(1-q^{a_0+a+1-a_t})}\right)\frac{(q)_{a+a_0}}{(q)_{a_0}(q)_{a_1}\cdots(q)_{a_n}},}
\end{align}}
where $\widetilde{L}(r,t)=2\sum_{k=t+1}^{r-1}a_k+\sum_{k=1}^{t-1}a_k$.
\end{cor}

When $m=2$ and $p_1=p_2=1$, we obtain
\begin{cor}[Conjecture 1.7, \cite{sills2006}]\label{conj-s3}
Let $r,s,t$ be fixed integers with $1\leq r< s \leq n, t<s$ and $n\geq 3$. Then
{\small\begin{align}\label{e-conj3}
&\CT_{\mathbf{x}}\,\frac{x_rx_s}{x_0x_t}\,D_n(\mathbf{x},\mathbf{a},q)\nonumber\\
&{\tiny=q^{\widetilde{L}(r,s,t)}\left(\frac{(1-q^{a_0})(1-q^{a_t})\Big((1-q^{a_0+a+1})+q^{M(r,s,t)}
(1-q^{a+1-a_t})\Big)}
{(1-q^{a+1-a_t})(1-q^{a+1})
(1-q^{a_0+a+1-a_t})}\right)\frac{(q)_{a+a_0}}{(q)_{a_0}(q)_{a_1}\cdots(q)_{a_n}},}
\end{align}}
where
\begin{equation*}
\widetilde{L}(r,s,t)=
\left\{%
\begin{array}{ll}
    \sum_{k=1}^{r-1}a_k+\sum_{k=t+1}^{s-1}a_k, &\ \hbox{if}\  \quad r<t<s;\\
    \sum_{k=r}^{s-1}a_k+\sum_{k=1}^{t-1}a_k+2\sum_{k=t+1}^{r-1}a_k, &\ \hbox{if}\ \quad  t<r<s, \\
\end{array}%
\right.
\end{equation*}
and
\begin{equation*}
M(r,s,t)=
\left\{%
\begin{array}{ll}
    1+a+a_0, &\ \hbox{if}\  \quad r<t<s;\\
    a_t, &\ \hbox{if}\ \quad  t<r<s. \\
\end{array}%
\right.
\end{equation*}
\end{cor}

When letting $q$ approach $1$ from the left, we get
\begin{thm}
Let $i_1,\ldots,i_m$ and $j_1,\ldots,j_{\nu}$ be distinct integers
with  $0= i_1<\cdots<i_m< n$ and $0< j_1<\cdots<j_{\nu}\leq n$. Then
\begin{equation*}{\small
\CT_{\mathbf{x}}\frac{x_{j_1}^{p_1}\cdots
x_{j_{\nu}}^{p_{\nu}}}{x_{i_1}x_{i_2}\cdots x_{i_{m}}} \prod_{0\le i\ne j
\le n} \left(1-\frac{x_i}{x_j}\right)^{\!\!a_i}
=\frac{(a_0+a_1+\cdots+a_n)!}{a_0!\, a_1!\, \cdots
a_n!}\sum_{\varnothing\neq T\subseteq
I_0}(-1)^{d}\frac{\sigma(T)}{1+a+a_0-\sigma(T)},}
\end{equation*}
where the $p$'s are positive integers with $\sum_{i=1}^{\nu}p_i=m$.
\end{thm}

The proof of Theorem \ref{t-main-thm} is along the same line of
Gessel and Xin's proof of Theorem \ref{thm-dyson}
\cite{gess-xin2006}, but with a major improvement. First of all, the
underlying idea is the well-known fact that proving the equality of
two polynomials of degree at most $d$, it suffices to prove that
they are equal at $d+1$ points. As is often the case, points at
which the polynomials vanish are most easily dealt with.

It is routine to show that after fixing parameters $a_1,\dots, a_n$,
the constant term is a polynomial of degree at most $d$ in the
variable $q^{a_0}$. Then we can apply the Gessel-Xin's technique to
show that the equality holds when the polynomial vanishes. The proof then
differs in showing the equality at the extra points: The $q$-Dyson
conjecture needs one extra point, which can be shown by induction;
Corollaries \ref{conj-s1}, \ref{conj-s2}, and \ref{conj-s3} need
one, two and two extra points respectively; Theorem \ref{t-main-thm}
needs many extra points. To prove Theorem \ref{t-main-thm}, we
develop, based on Gessel and Xin's work, a new technique in
evaluating the constant terms at these extra points.

This paper is organized as follows. In section 2, our main result,
Theorem \ref{t-main-thm}, is established under the assumption of two
main lemmas. The first lemma is for the vanishing points and the
second one is for the extra points, and they take us the next three
sections to prove. Then by specializing our main theorem, we prove
Sills' three conjectures. In section 3, we introduce the field of
iterated Laurent series and partial fraction decompositions as basic
tools for evaluating constant terms. We also introduce basic notions
and lemmas of \cite{gess-xin2006} in a generalized form. These are
essential for proving the two main lemmas. In section 4, we deal
with some general $q$-Dyson style constant terms and prove our first
main lemma. Section 5 includes new techniques and complicated
computations for our second main lemma. It is a continuation of
section 4.

\section{The proofs and the consequences}
Dyson's conjecture, Andrews' $q$-Dyson conjecture, and their
relatives are all constant terms of certain Laurent polynomials.
However, larger rings and fields will encounter when evaluating
them. We closely follow the notation in \cite{gess-xin2006}. In
order to prove our Main Theorem, we make several generalizations
that need to go into details to explain.

We first work in the ring of Laurent polynomials to see that some
seemingly more complicated cases can be solved by Theorem
\ref{t-main-thm}.

Define an action $\pi$ on Laurent polynomials by
\begin{align*}
\pi \big(F(x_0,x_1,\ldots,x_n)\big)=F(x_1,x_2,\dots,x_n,x_0/q).
\end{align*}
By iterating, if $F(x_0,x_1,\dots,x_n)$ is homogeneous of degree
$0$, then
$$\pi^{n+1}\big(F(x_0,x_1,\ldots,x_n)\big)=F(x_0/q,x_1/q,\ldots,x_n/q)=F(x_0,x_1,\ldots,x_n),  $$
so that in particular $\pi$ is a cyclic action on
$D_n(\mathbf{x},\mathbf{a},q)$.

\begin{lem}\label{l-pi}
Let $L(\mathbf{x})$ be a Laurent polynomial
in the $x$'s. Then
\begin{align}\label{e-cyclic}
\CT_\mathbf{x} L(\mathbf{x})\,D_n(\mathbf{x},\mathbf{a},q)=
\CT_\mathbf{x} \, \pi
\big(L(\mathbf{x})\big)D_n\big(\mathbf{x},(a_n,a_0,\ldots,a_{n-1}),q\big).
\end{align}
By iterating \eqref{e-cyclic} and renaming the parameters,
evaluating $\CT_\mathbf{x}
L(\mathbf{x})\,D_n(\mathbf{x},\mathbf{a},q) $ is equivalent to
evaluating $\CT_\mathbf{x}
\pi^k(L(\mathbf{x}))\,D_n(\mathbf{x},\mathbf{a},q)$ for any integer
$k$.
\end{lem}
\begin{proof}
It is straightforward to check that
$$\pi \big(D_n(\mathbf{x},\mathbf{a},q)\big)= D_n\big(\mathbf{x},(a_n,a_0,\ldots,a_{n-1}),q\big).$$
Note that an equivalent form was observed by Kadell
\cite[Equation~5.12]{kadell1998}.  Therefore, equation
\eqref{e-cyclic} follows by the above equality and the fact
$$\CT_\mathbf{x} F(x_0,x_1,\dots,x_n) =
\CT_\mathbf{x} \pi\big( F(x_0,x_1,\dots,x_n)\big).$$ 
The second part of the lemma is obvious.
\end{proof}

Next we work in the ring of Laurent series in $x_0$ with coefficients
Laurent polynomials in $x_1,x_2,\dots,x_n$. The following lemma is
a generalized form of Lemma 3.1 in \cite{gess-xin2006}. The proof
is similar.
\begin{lem}\label{lem1}
Let $L(x_1,\ldots,x_n)$ be a Laurent polynomial independent of
${a_0}$ and $x_0$. Then for fixed nonnegative integers $a_1,\ldots,a_n$
and $k\leq a$, $k\in \mathbb{Z}$ the constant term
\begin{align}\label{p1}
\CT_\mathbf{x} x_0^k L(x_1,\dots,x_n) D_n(\mathbf{x},\mathbf{a},q)
\end{align}
is a polynomial in $q^{a_0}$
of degree at most $a-k$.
\end{lem}
\begin{proof}
It is easy to prove that
\begin{align*}
\left(\frac{x_0}{x_j} \right)_{\!\!a_0}\!\left(\frac{x_j}{x_0}q
\right)_{\!\!a_j}&=q^{\binom{a_j+1}{2}}\left(-\frac{x_j}{x_0}\right)^{\!\!a_j}
\!\left(\frac{x_0}{x_j} q^{-a_j}\right)_{\!\!a_0+a_j}
\end{align*}
for all integers $a_0$, where both sides are regarded as Laurent
series in $x_0$. Rewrite \eqref{p1} as
\begin{align}
\label{e-product} \CT_{\mathbf{x}}\: x_0^kL_1(x_1,\dots
,x_n) \prod_{j=1}^n q^{\binom{a_j+1}{2}}
\left(-\frac{x_j}{x_0}\right)^{\!\!a_j}\! \left(\frac{x_0}{x_j}
q^{-a_j}\right)_{\!\!a_0+a_j}\! ,
\end{align}
where $L_1(x_1,\dots,x_n)$ is a Laurent polynomial in
$x_1,\dots, x_n$ independent of $x_0$ and $a_0$.

The well-known $q$-binomial theorem \cite[Theorem
2.1]{andrew-qbinomial} is the identity
\begin{align}
\label{e-qbinomial} \frac{(bz)_\infty}{(z)_\infty} =
\sum_{k=0}^\infty \frac{(b)_k}{(q)_k} z^k.
\end{align}
Setting $z=uq^n$ and $b=q^{-n}$ in \eqref{e-qbinomial}, we obtain
\begin{align}\label{e-qbinomialn}
(u)_n=\frac{(u)_\infty}{(uq^n)_\infty}=
\sum_{k=0}^\infty q^{k(k-1)/2}\qbinom{n}{k} (-u)^k
\end{align}
for all integers $n$, where $\qbinom{n}{k}=\frac{(q)_n}{(q)_k(q)_{n-k}}$ is the $q$-binomial coefficient.

Using \eqref{e-qbinomialn}, we see
that for $1\le j\le n$, {\small
\begin{align*} q^{\binom{a_j+1}{2}}
\left(-\frac{x_j}{x_0}\right)^{\!\!a_j} \left(\frac{x_0}{x_j}
  q^{-a_j}\right)_{\!\!a_0+a_j}
=\sum_{k_j\geq 0}C(k_j) \qbinom{a_0+a_j}{k_j}
x_0^{k_j-a_j}x_j^{a_j-k_j},
\end{align*}}
where $C(k_j)=(-1)^{k_j+a_j}q^{\binom{a_j+1}2 + \binom {k_j}2
-k_ja_j}$.

Expanding the product in \eqref{e-product} and  taking constant
term in $x_0$, we see that \eqref{p1} becomes
{\small\begin{align} \label{e-midle} \sum_{\mathbf{k}}
\qbinom{a_0+a_1}{k_1}\qbinom{a_0+a_2}{k_2}\cdots
\qbinom{a_0+a_n}{k_n} \CT_{x_1,\dots,x_n} L_2(x_1,\dots
,x_n;\mathbf{k}),
\end{align}}
where $L_2(x_1, \dots, x_n;\mathbf{k})$ is a Laurent
polynomial in $x_1, \dots, x_n$ independent of $a_0$ and the sum
ranges over all sequences $\mathbf{k}=(k_1,\dots, k_n)$ of
nonnegative integers satisfying $k_1+k_2+\cdots+k_n=a-k.$ Since
$\qbinom{a_0+a_i}{k_i}$ is a polynomial in $q^{a_0}$ of degree
$k_i$, each summand in \eqref{e-midle} is a polynomial in $q^{a_0}$
of degree at most $k_1+k_2+\cdots +k_n=a-k$, and so is the sum.
\end{proof}

Lemma \ref{lem1} reduces the proof of Theorem \ref{t-main-thm} to
evaluating the constant term at enough values of the $q^{a_0}$'s.
This is accomplished by the following Main Lemmas 1 and 2. Their
proofs will be given in the next three sections, using the field
of iterated Laurent series \cite{xiniterate}.

\begin{lem}[Main Lemma 1]\label{lem-main1}
If $a_0$ belongs to the set $\{0,-1,\ldots,-(a+1)\}\setminus
\{-(a-\sigma(T)+1)\mid T\subseteq I\}$, then
\begin{align}\label{main-lem}
\CT_{\mathbf{x}}\frac{x_{j_1}^{p_1}\cdots x_{j_{\nu}}^{p_{\nu}}}{x_{i_1}x_{i_2}\cdots x_{i_m}}D_n(\mathbf{x},\mathbf{a},q) =0.
\end{align}
\end{lem}

\begin{lem}[Main Lemma 2]\label{lem-main2}
If $a_0$ belongs to the set $\{-(a-\sigma(T)+1)\mid T\subseteq
I\}$, then
\begin{align}\label{e-mainlem2}
\CT_{\mathbf{x}}\frac{x_{j_1}^{p_1}\cdots x_{j_{\nu}}^{p_{\nu}}}{x_{i_1}x_{i_2}\cdots x_{i_m}}D_n(\mathbf{x},\mathbf{a},q) =
\sum_{T}(-1)^{w+d}q^{L^*(T)}\frac{(q)_w(q)_{a-w}}{(q)_{a_1}\cdots
(q)_{a_n}},
\end{align}
where
the sum
ranges over all $T\subseteq I$ such that $-(a-\sigma(T)+1)=a_0$
and
\begin{align}\label{e-LT*}{\small
L^*(T)=\sum_{l\in I}\sum_{i=l}^nw_i-\sum_{l=1}^{\nu}p_l\sum_{i=j_l}^nw_i-{w+1\choose 2}-1.}
\end{align}
\end{lem}

The following lemma shows that Main Lemmas 1 and 2 coincide with our Main Theorem.
\begin{lem}\label{main2-1}
If $a_0$ belongs to the set $\{-(a-\sigma(T)+1)\mid T\subseteq I\}$,
then {\small\begin{align}\label{lem-main6}
\frac{(q)_{a+a_0}}{(q)_{a_0}(q)_{a_1}\cdots(q)_{a_n}}\sum_{\varnothing\ne
T\subseteq I_0}(-1)^dq^{L(T)}
\frac{1-q^{\sigma(T)}}{1-q^{1+a_0+a-\sigma(T)}}
=\sum_{T}(-1)^{w+d}q^{L^*(T)}\frac{(q)_w(q)_{a-w}}{(q)_{a_1}\cdots
(q)_{a_n}},
\end{align}}
where the last sum ranges over all $T\subseteq I$ such that
$-(a-\sigma(T)+1)=a_0$, $L^*(T)$ is defined as in \eqref{e-LT*}, and
$L(T)$ is defined as in \eqref{e-LT}.

If $a_0$ belongs to the set $\{0,-1,\ldots,-(a+1)\}\setminus
\{-(a-\sigma(T)+1)\mid T\subseteq I\}$, then the left-hand side of
\eqref{lem-main6} vanishes.
\end{lem}
\begin{proof}
Let $LHS$ and $RHS$ denote the left-hand side and the right-hand
side of \eqref{lem-main6} respectively. By definition, $L(T)=L(T\cup
\{0\})+a_0$ for any $T\subseteq I$. This fact will be used.

If $a_0=0$, then simplifying gives
$$LHS=\frac{(q)_a}{(q)_{a_1}\cdots (q)_{a_n}} \sum_{T\subseteq I_0}(-1)^dq^{L(T)}
\frac{1-q^{\sigma(T)}}{1-q^{1+a-\sigma(T)}},   $$ where we have added the
vanishing term corresponding to $T=\varnothing $. The sum equals $0$ since for
every $T \subseteq I$, when pairing the summand for $T$ and the
summand for $T\cup \{0\}$, we have
$$(-1)^dq^{L(T)}
\frac{1-q^{\sigma(T)}}{1-q^{1+a-\sigma(T)}}+(-1)^{d+1}q^{L(T\cup
\{0\})} \frac{1-q^{\sigma(T\cup \{0\})}}{1-q^{1+a-\sigma(T\cup
\{0\})}}=0.$$

If $a_0=-a-1$, then the sum for $RHS$ has only one term
corresponding to $T=\varnothing$. For $LHS$, simplifying gives
$$
LHS=\frac{(q)_{-1}}{(q)_{-a-1}(q)_{a_1}\cdots(q)_{a_n}}\sum_{\varnothing\ne
T\subseteq I_0}(-1)^{d+1}q^{L(T)+\sigma(T)}.
$$
Since for any $T\subseteq I$, we have
\begin{multline*}
\qquad \qquad\qquad \qquad(-1)^{d+1}q^{L(T)+\sigma(T)}+(-1)^{d+2}q^{L\big(T\cup \{0\}\big)+\sigma\big(T\cup \{0\}\big)}\\
=(-1)^{d+1}q^{\big(L(T)+\sigma(T)\big)}\big(1-q^{-a_0+a_0}\big)=0,
\qquad \qquad\qquad \qquad
\end{multline*}
$LHS$ reduces to only one term corresponding to $T=\{0\}$,
which is {\small\begin{align*}
LHS=&(-1)^{2}q^{L(\{0\})+a_0}\frac{(q)_{-1}}{(q)_{-a-1}(q)_{a_1}\cdots(q)_{a_n}}
=q^{L(\{0\})+a_0}\frac{\big(1-\frac{1}{q}\big)\cdots\big(1-\frac{1}{q^a}\big)}{(q)_{a_1}\cdots(q)_{a_n}}\\
=&(-1)^aq^{L(\{0\})-a-1-{a+1\choose 2}}\frac{(q)_a}{(q)_{a_1}\cdots(q)_{a_n}}
=(-1)^aq^{L^{*}(\varnothing)}\frac{(q)_a}{(q)_{a_1}\cdots(q)_{a_n}}=RHS.
\end{align*}}

Now consider the cases $a_0=-1,\dots,-a$. Since the factor
$(q)_{a_0+a}/(q)_{a_0}=(1-q^{a_0+1})\cdots (1-q^{a_0+a})$ of $LHS$
vanishes for $a_0=-1,-2,\dots,-a$, the summand with respect to $T$
has no contribution unless the denominator
$1-q^{1+a_0+a-\sigma(T)}=0$, i.e., $a_0=-\big(a+1-\sigma(T)\big)$.
Therefore, $LHS=0$ if $a_0$ does not belong to
$\{-(a-\sigma(T)+1)\mid T\subseteq I\}$. If it is not the case, then
only those terms with $-(a-\sigma(T)+1)=a_0$ have contributions.
Such $T$ can not contain $0$, for otherwise we may deduce that
$a+1-\sigma(T\setminus \{0\})=0$, which is impossible. Therefore it
suffices to show that for every subset $T\subseteq I$ we have
{\small\begin{align}\label{e-LL}
\frac{(q)_{a+a_0}}{(q)_{a_0}\cdots(q)_{a_n}}(-1)^dq^{L(T)}&
\frac{1-q^{\sigma(T)}}{1-q^{1+a_0+a-\sigma(T)}}\Big|_{a_0=-w-1}
=&\frac{(q)_w(q)_{a-w}}{(q)_{a_1}\cdots
(q)_{a_n}}(-1)^{w+d}q^{L^*(T)}.
\end{align}}
Since $L(T)|_{a_0=-w-1}=L^*(T)+{w+1\choose 2}$, the left-hand side
of \eqref{e-LL} equals {\small\begin{align*}
(-1)&^dq^{L^*(T)+{w+1\choose
2}}\frac{\big[(1-q^{-w})\cdots(1-q^{-1})\big]\big[(1-q)\cdots
(1-q^{a-w})\big]}{(q)_{a_1}\cdots (q)_{a_n}}
=(-1)^{w+d}q^{L^*(T)}\frac{(q)_w(q)_{a-w}}{(q)_{a_1}\cdots(q)_{a_n}},
\end{align*}}
which is the right-hand side of \eqref{e-LL}.
\end{proof}

\begin{proof}[Proof of Theorem {\bf\ref{t-main-thm}}]
We prove the theorem by showing that both sides of \eqref{e-main}
are polynomials in $q^{a_0}$ of degree no more than $a+1$, and that
they agree at the $a+2$ values corresponding to
$a_{0}=0,-1,\ldots,-a-1$. The latter statement follows by Main Lemma
1, Main Lemma 2, and Lemma \ref{main2-1}. We now prove the former
statement to complete the proof.

 Applying
Lemma \ref{lem1} in the case $k=-1$ and
$L(x_1,\ldots,x_n)=x_{j_1}^{p_1}\cdots x_{j_{\nu}}^{p_{\nu}}/({x_{i_2}\cdots
x_{i_m}})$, we see that the constant term in \eqref{e-main} is a
polynomial in $q^{a_0}$ of degree at most $a+1$. The right-hand side
of \eqref{e-main} can be written as
$$\sum_{\varnothing \ne T \subseteq I_0} (-1)^d q^{L(T)}
\frac{1-q^{\sigma(T)}}{1-q^{a_0+1+a-\sigma(T)}}\frac{(1-q^{a_0+1})(1-q^{a_0+2})
\cdots (1-q^{a_0+a})}{(q)_{a_1}(q)_{a_2}\cdots (q)_{a_n}}. $$ This
is a polynomial in $q^{a_0}$ of degree no more than $a+1$, as can be seen by checking the two cases:
If $0\not \in T$ then the degree of $q^{L(T)}$ in $q^{a_0}$ is
$1$ and $1-q^{a_0+1+a-\sigma(T)}$ cancels with the numerator so that
the summand has degree $a$ in $q^{a_0}$; Otherwise the summand has
degree $a+1$ in $q^{a_0}$.
\end{proof}

The $m=0$ case of Theorem \ref{t-main-thm} reduces to the
Zeilberger-Bressoud $q$-Dyson Theorem. Comparing with the proof of
Theorem \ref{thm-dyson} in \cite{gess-xin2006}, the new part is
Lemma \ref{lem-main2}, where we give explicit formula for the
non-vanishing case $a_0=-a-1$. This gives a proof without using
induction on $n$.

\begin{proof}[Proof of Corollary {\bf\ref{conj-s1}}]
Applying the Main Theorem for $I_{0}=\{0\}$ gives
\begin{align*}
L(\{0\})=\sum_{i=0}^nw_i-\sum_{i=r}^nw_i=\sum_{i=1}^na_i-\sum_{i=r}^na_i=\sum_{i=1}^{r-1}a_i.
\end{align*}
Substituting the above into \eqref{e-main} and simplifying, we
obtain Corollary \ref{conj-s1}.
\end{proof}

\begin{proof}[Proof of Corollary {\bf\ref{conj-s2}}]
Applying the Main Theorem for $I_{0}=\{0,t\}$ and $p_1=2$ gives
{\small\begin{align*}
L(\{0\})=&\sum_{i=1}^na_i+\sum_{i=t}^na_i-2\sum_{i=r}^na_i, \\
L(\{t\})=&\sum_{i=0}^na_i+\sum_{i=t}^na_i-2\sum_{i=r}^na_i-2a_t, \\
L(\{0,t\})=&\sum_{i=1}^na_i+\sum_{i=t}^na_i-2\sum_{i=r}^na_i-2a_t.
\end{align*}}
Substituting the above into \eqref{e-main} and simplifying,
we obtain Corollary \ref{conj-s2}.
\end{proof}

\begin{proof}[Proof of Corollary {\bf\ref{conj-s3}}]
Applying the Main Theorem for $I_{0}=\{0,t\}$ and $p_1=p_2=1$ gives
{\begin{align*}
L(\{0\})=&\sum_{i=1}^na_i+\sum_{i=t}^na_i-\sum_{i=r}^na_i-\sum_{i=s}^na_i, \\
L(\{t\})=&\left\{%
\begin{array}{ll}
    \sum_{i=0}^na_i+\sum_{i=t}^na_i-\sum_{i=r}^na_i-\sum_{i=s}^na_i-a_t, &\ \hbox{if}\ \ r<t<s, \\
\sum_{i=0}^na_i+\sum_{i=t}^na_i-\sum_{i=r}^na_i-\sum_{i=s}^na_i-2a_t, &\ \hbox{if}\ \ t<r<s ,\\\end{array}%
\right. \\
L(\{0,t\})=&\left\{%
\begin{array}{ll}
    \sum_{i=1}^na_i+\sum_{i=t}^na_i-\sum_{i=r}^na_i-\sum_{i=s}^na_i-a_t, &\ \hbox{if}\ \ r<t<s, \\
\sum_{i=1}^na_i+\sum_{i=t}^na_i-\sum_{i=r}^na_i-\sum_{i=s}^na_i-2a_t, &\ \hbox{if}\ \ t<r<s .\\\end{array}%
\right.
\end{align*}}
Substituting the above into \eqref{e-main} and simplifying,
we obtain Corollary \ref{conj-s3}.
\end{proof}

\section{Constant term evaluations and basic lemmas}

From now on, we let $K=\CC(q)$, and assume that all series are in
the field of iterated Laurent series $K\langle\!\langle x_n,
x_{n-1},\ldots,x_0\rangle\!\rangle
=K(\!(x_n)\!)(\!(x_{n-1})\!)\cdots (\!(x_0)\!)$. This means that all
series are regarded first as Laurent series in $x_0$, then as
Laurent series in $x_1$, and so on. The reason for choosing
$K\langle\!\langle x_n, x_{n-1},\ldots,x_0\rangle\!\rangle$ as a
working field has been explained in \cite{gess-xin2006}. For more
detailed account of the properties of this field, with other
applications, see \cite{xinresidue} and \cite{xiniterate}.

We emphasize that the field of rational functions is a subfield
of $K\langle\!\langle x_n, x_{n-1},\ldots,x_0\rangle\!\rangle$, so
that every rational function is identified with its unique iterated
Laurent series expansion. The series expansions of $1/(1-q^k
x_i/x_j)$ will be especially important. If $i<j$ then
$$\frac{1}{1-q^k x_i/x_j}=\sum_{l= 0}^\infty q^{kl} x_i^l x_j^{-l}.$$
However, if $i>j$ then this expansion is not valid and instead we
have the expansion
{\small$$ \frac{1}{1-q^k x_i/x_j}=\frac1{-q^k x_i/x_j(1-q^{-k}x_j/x_i)}
    =\sum_{l=0}^\infty -q^{-k(l+1)} x_i^{-l-1}x_j^{l+1}.$$}

The constant term of the series $F(\mathbf{x})$ in $x_i$,
denoted by $\CT_{x_i} F(\mathbf{x})$, is defined to be the sum of
those terms in $F(\mathbf{x})$ that are free of $x_i$.
It follows that
\begin{equation}
\label{e-ct} \CT_{x_i} \frac{1}{1-q^k x_i/x_j} =
\begin{cases}
    1, & \text{ if }i<j, \\
    0, & \text{ if }i>j. \\
\end{cases}
\end{equation}
We shall call the monomial $M=q^k x_i/x_j$ \emph{small} if $i<j$ and
\emph{large} if $i>j$.  Thus the constant term in $x_i$ of $1/(1-M)$
is $1$ if $M$ is small and $0$ if $M$ is large.

An important property of the constant term operators defined in this
way is their commutativity:
$$\CT_{x_i} \CT _{x_j} F(\mathbf{x}) = \CT_{x_j} \CT_{x_i} F(\mathbf{x}).$$
Commutativity implies that the constant term in a set of variables
is well-defined, and this property will be used in our proof of the
two Main Lemmas. (Note that, by contrast, the constant term
operators in \cite{zeil} do not commute.)

The \emph{degree} of a rational function of $x$ is the degree in $x$
of the numerator minus the degree  in $x$ of the denominator. For
example, if $i\ne j$ then  the degree  of $1-x_j/x_i=(x_i-x_j)/x_i$
is $0$ in $x_i$ and $1$ in $x_j$. A rational function is called
\emph{proper} (resp. \emph{almost proper}) in $x$ if its degree in
$x$ is negative (resp. zero).

Let
\begin{align}\label{e-defF}
F=\frac{p(x_k)}{x_k ^d \prod_{i=1}^m (1-x_k/\alpha_i)}
\end{align}
be a rational function of $x_k$, where $p(x_k)$ is a polynomial in
$x_k$, and the $\alpha_i$ are distinct monomials, each of the form
$x_t q^s$. Then the partial fraction decomposition of $F$ with
respect to $x_k$ has the following form:
{\small\begin{align}\label{e-defFs}
F=p_0(x_k)+\frac{p_1(x_k)}{x_k^d}+\sum_{j=1}^m \frac{1}{1-
x_k/\alpha_j}  \left. \left(\frac{p(x_k)}{x_k^d \prod_{i=1,i\ne j}^m
(1-x_k/\alpha_i)}\right)\right|_{x_k=\alpha_j},
\end{align}}
where $p_0(x_k)$
is a polynomial in $x_k$, and $p_1(x_k)$ is a polynomial in $x_k$ of
degree less than $d$.

The following lemma is the
basic tool in extracting constant terms.
\begin{lem}\label{lem-almostprop}
Let $F$ be as in \eqref{e-defF} and \eqref{e-defFs}.
 Then
\begin{align}\label{e-almostprop}
\CT_{x_k} F=p_0(0) +\sum_j  \bigl(F\,
(1-x_k/\alpha_j)\bigr)\Bigr|_{x_k =\alpha_j},
\end{align}
where 
the sum ranges over all $j$ such that $x_k/\alpha_j$ is small. In
particular, if $F$ is proper in $x_k$, then $p_0(x_k)=0$; if $F$ is
almost proper in $x_{k}$, then
$p_0(x_k)=(-1)^m\prod_{i=1}^m\alpha_{i}\LC_{x_{k}}p(x_k)$, where
$\LC_{x_k}$ means to take the leading coefficient with respect to
$x_k$.
\end{lem}
Lemma \ref{lem-almostprop} is the general form of \cite[Lemma
4.1]{gess-xin2006} and the proof is also straightforward. The new
observation is that we have explicit formulas not only for proper
$F$ but also for almost proper $F$. Such explicit formulas are
useful in predicting the final result when iterating Lemma
\ref{lem-almostprop}.

The following slight generalization of \cite[Lemma
4.2]{gess-xin2006} plays an important role in our argument.
\begin{lem}\label{lem-import}
Let $a_{1},\ldots,a_{s}$ be nonnegative integers. Then for any
positive integers $k_{1},\ldots,k_{s}$ with $1\leq k_{i}\leq
a_{1}+\cdots+a_{s}+1$ for all $i$, either $1\leq k_{i}\leq a_{i}$
for some $i$ or $-a_{j}\leq k_{i}-k_{j}\leq a_{i}-1$ for some $i<j$,
except only when $k_{i}=a_{i}+\cdots+a_{s}+1$ for $i=1,\ldots,s$.
\end{lem}
\begin{proof}
The basic idea is the same as of \cite[Lemma 4.2]{gess-xin2006}.
Assume $k_1,\dots ,k_s$ to satisfy that for all $i$, $a_i <
k_i\le a_1+\cdots +a_s+1$, and for all $i<j,$ either $k_i-k_j\ge
a_i$ or $k_i-k_j\le -a_j-1$. Then we need to show that
$k_{i}=a_{i}+\cdots+a_{s}+1$ for $i=1,\ldots,s$.

We construct a tournament on $1,2,\dots ,s$ with numbers on the arcs
as follows: {}For $i<j$, if $k_i-k_j\ge a_i$ then we draw an arc
$i\mathop{\longleftarrow}\limits^{a_i} j$ from $j$ to $i$ and if
$k_i-k_j\le -1-a_j$ then we draw an arc
$i\mathop{\longrightarrow}\limits^{a_j+1} j$ from $i$ to $j$.

We call an arc from $u$ to $v$ an \emph{ascending arc} if $u<v$ and
a \emph{descending arc} if $u>v$.  We note two facts: (i) the number
on an arc from $u$ to $v$ is less than or equal to $k_v-k_u$, and
(ii) the number on an ascending arc is always positive.

A  consequence of (i) is that for any directed path from $e$ to $f$,
the sum along the arcs is less than or equal to $k_f-k_e$. It
follows that the sum along a cycle is non-positive. But any cycle
must have at least one ascending arc, and by (ii) the number on this
arc is positive, and so the sum along the cycle is positive. Thus
there can be no cycles.

Therefore the tournament we have constructed is transitive, and
hence defines a total ordering $\rightarrow$ on $1,2,\dots ,s$.
Assume the total ordering is given by $i_1\rightarrow i_2\rightarrow
\cdots \rightarrow i_{s-1}\rightarrow i_s$. Then $k_{i_s}-k_{i_1}\ge
a_{i_2}+a_{i_3}+\cdots +a_{i_s}$. This implies that
\begin{align}\label{e-contradiction}
k_{i_s}&\ge k_{i_1}+a_{i_2}+a_{i_3}+\cdots+a_{i_s}\nonumber\\
    & \ge a_{i_1}+1+a_{i_2}+a_{i_3}+\cdots +a_{i_s}\nonumber\\
    &=a_1+a_2+\cdots +a_s+1,
\end{align}
By assumption, $1\leq k_{i}\leq a_{1}+\cdots+a_{s}+1$ for all $i$,
so $k_{i_s}=a_1+a_2+\cdots +a_s+1$. But for the equality in
\eqref{e-contradiction} to hold, we must have $k_{i_1}=a_{i_1}+1$,
and there are no arcs of the form
$i_{l-1}\mathop{\longrightarrow}\limits^{a_{i_{l}}+1} i_{l}$ (i.e.,
$i_{l-1}<i_l $) for $l=2,3,\dots,s$.
It follows that the total ordering $i_1\rightarrow i_2\rightarrow
\cdots \rightarrow i_{s-1}\rightarrow i_s$ is actually $s\rightarrow
(s-1)\rightarrow \cdots \rightarrow 2\rightarrow 1$. One can then
deduce that
\begin{align*} k_{i_{l}}=a_{i_{1}}+\cdots+a_{i_{l}}+1,
\quad \mbox{for}\quad l=1,\ldots,s.
\end{align*}
This completes our proof.
\end{proof}

\section{The general setup and the proof of Main Lemma 1}

 Fix a monomial $M(\mathbf{x})=\prod_{i=0}^{n}x_i^{b_i}$
with $\sum_{i=0}^{n} b_i=0$. We derive general properties for
$q$-Dyson style constant terms, and specialize $M(\mathbf{x})$ for
the proofs of our main lemmas.

Define $Q(h)$ to be
\begin{align}\label{qh}
Q(h):=M(\mathbf{x}) \prod_{j=1}^{n}
\left(\frac{x_0}{x_j}\right)_{\!\!\!-h}\left(\frac{x_j}{x_0}q\right)_{\!\!\!a_j}
 \prod_{1\leq i<j\leq n}
 \left(\frac{x_i}{x_j}\right)_{\!\!\!a_i}\left(\frac{x_j}{x_i}q\right)_{\!\!\!a_j}.
\end{align}
If $h\ge 0$, then
\begin{align}\label{qh1}
Q(h)=\prod_{i=0}^{n}x_i^{b_i}\prod_{j=1}^{n} \frac{(x_jq/x_0)_{a_j}}
{\big(1-\frac{x_0}{x_jq}\big)\big(1-\frac{x_0}{x_jq^2}\big)\cdots
\big(1-\frac{x_0}{x_jq^h}\big)}
 \prod_{1\leq i<j\leq n} \left(\frac{x_i}{x_j}\right)_{\!\!\!a_i}\left(\frac{x_j}{x_i}q\right)_{\!\!\!a_j}.
\end{align}
We are interested in the constant term of $Q(h)$ for
$h=0,1,2,\dots,a+1$.

Since the degree in $x_0$ of $1-x_jq^i/x_0$ is zero, the degree in
$x_0$ of $Q(h)$ is  $b_0-nh$. Thus when $h>\frac{b_0}{n}$, $Q(h)$ is
proper. Applying Lemma \ref{lem-almostprop}, we have
\begin{align}\label{qh2}
\CT_{x_0}Q(h)=\sum_{\substack{0<r_1\leq n,\\ 1\leq k_1\leq
h}}Q(h\mid r_1;k_1),
\end{align}
where
$$
Q(h\mid r_1;k_1)=Q(h)\left(1-\frac{x_0}{x_{r_1}q^{k_1}}\right)\bigg|_{x_0=x_{r_1}q^{k_1}}.
$$
For each term in \eqref{qh2} we will extract the constant term in
$x_{r_1}$, and then perform further constant term extractions,
eliminating one variable at each step. In order to keep track of the
terms we obtain, we introduce some notations from
\cite{gess-xin2006}.

For any rational function $F$ of $x_0, x_1, \ldots , x_n$, and for
sequences of integers $\mathbf{k} = (k_1,\ldots, k_s)$ and
$\mathbf{r} = (r_1, r_2,\ldots, r_s)$ let $E_{\mathbf{r,k}} F$ be
the result of replacing $x_{r_i}$ in $F$ with $x_{r_s}q^{k_s-k_i}$
for $i = 0, 1,\ldots , s-1$, where we set $r_0 = k_0 = 0$. Then for
$0 < r_1 < r_2 < \cdots < r_s \leq n$ and $0 < k_i \leq h$, we
define
\begin{align}\label{qh3}
Q(h\mid\mathbf{r;k})=Q(h\mid r_1,\ldots,r_s;k_1,\ldots,k_s)=
E_{\mathbf{r,k}}\left[Q(h)\prod_{i=1}^{s}\Big(1-\frac{x_0}{x_{r_i}q^{k_i}}\Big)\right].
\end{align}
Note that the product on the right-hand side of \eqref{qh3} cancels
all the factors in the denominator of $Q$ that would be taken to
zero by $E_{\mathbf{r,k}}$.

\begin{lem}\label{lem-lead1}
Let $R=\{r_0,r_1,\dots,r_s\}$. Then the rational functions
$Q(h\mid \mathbf{r;k})$ have the following two properties:
\begin{itemize}
 \item[{\bf i}] If $1\leq k_i\leq a_{r_1}+\cdots+a_{r_s}$ for all $i$ with $1\leq i\leq s$ and $h>\frac{b_0}{n}$,
then $ Q(h\mid \mathbf{r;k})=0$.
\item[{\bf ii}] If $k_i>a_{r_1}+\cdots+a_{r_s}$ for some $i$ with $1
\leq i \leq s<n$, and if
\begin{align}\label{deg1}
h>a_{r_1}+\cdots +a_{r_s}+\frac{\sum_{{i\in R}}b_i}{n-s},
\end{align}
then
\begin{align}\label{qh4}
\CT_{x_s}Q(h\mid \mathbf{r;k})= \sum_{\substack{r_s<r_{s+1}\leq n,\\
1\leq k_{s+1}\leq h}}
Q(h\mid r_1,\ldots,r_s,r_{s+1};k_1,\ldots,k_s,k_{s+1}).
\end{align}
\end{itemize}
\end{lem}
\begin{proof}
[Proof of property {\bf(i)}] By Lemma \ref{lem-import}, either $1\le
k_i\le a_{r_i}$ for some $i$ with $1\le i \le s$, or $-a_{r_j}\le
k_i-k_j\le a_{r_i}-1$ for some $i<j$, since the exceptional case can not happen. If $1\le k_i\le a_{r_i}$ then
$\Qrk$ has the factor
$$\Erk
\left[\lrq{x_{r_i}}{x_{0}}{q}_{\!\!a_{r_i}} \right]
=\lrq{x_{r_s}q^{k_s-k_i}}{x_{r_s}q^{k_s}}{q}_{\!\!a_{r_i}}
=(q^{1-k_i})_{a_{r_i}}=0.$$

If $-a_{r_j}\le k_i-k_j\le a_{r_i}-1$ where $i<j$ then $\Qrk$ has
the factor
$$\Erk\, \left[\lrq{x_{r_i}}{x_{r_j}}{}_{\!\!a_{r_i}}\!\!\lrq{x_{r_j}}{x_{r_i}}{q}_{\!\!a_{r_j}}\right], $$
which is equal to
$$q^{\binom{a_{r_j}+1}{2}} \left(-\frac{x_{r_j}}{x_{r_i}}\right)^{a_{r_j}}\!
\lrq{x_{r_i}}{x_{r_j}}{q^{-a_{r_j}}}_{\!\!a_{r_i}+a_{r_j}}
\!\!\!\!=q^{\binom{a_{r_j}+1}{2}}
(-q^{k_i-k_j})^{a_{r_j}}(q^{k_j-k_i-a_{r_j}})_{a_{r_i}+a_{r_j}}=0.
$$
\smallskip
\noindent\emph{Proof of property {\bf(ii).}} Note that since $h\ge k_i$
for all $i$, the hypothesis implies that   $h>a_{r_1}+\cdots
+a_{r_s}$.

We first show that $Q(h\mid \mathbf{r;k})$ is proper in $x_{r_s}$. To do
this we write $\Qrk$ as $N/D$, in which $N$ (the ``numerator") is
$$\Erk\, \left[\prod_{i=0}^nx_i^{b_i}\prod_{j=1}^n\lrq{x_j}{x_0}q_{\!\!a_j}
\cdot \prod_{\substack{1\le i, j\le n\\ j\neq i}}
\left(\frac{x_i}{x_j}\,q^{\chi(i>j)}\right)_{\!\!a_i}\right],$$ and
$D$ (the ``denominator") is
$$\Erk\, \left[\prod_{j=1}^n \lrq{x_0}{x_jq^h}{}_{\!\!\!h}\biggm/\prod_{i=1}^s\left(1-\frac
{x_0}{x_{r_i}q^{k_i}}\right)\right],$$
where $\chi(S)$ is $1$ if the statement $S$ is true, and $0$ otherwise. Notice that
$R=\{r_0,r_1,\dots,r_s\}$. Then the degree in $x_{r_s}$ of
$$\Erk\, \left[\left(1-\frac{x_i}{x_j}q^m\right)\right]$$
is 1 if $i\in R$ and $j\not\in R$, and is $0$ otherwise, as is
easily seen by checking the four cases. Clearly the degree in $x_{r_s}$
of $\Erk\, x_i^{b_i}$ is $b_i$ if $i\in R$ and is $0$ otherwise.
Thus the parts of $N$ contributing to the degree in
$x_{r_s}$ are

$$E_{\mathbf{r},\mathbf{k}}\left[\prod_{i\in R}x_i^{b_i}\prod_{i=1}^s \prod_{j\ne r_0,\dots ,r_s}
\left(\frac{x_{r_i}}{x_j}q^{\chi(r_i>j)}\right)_{\!\!a_{r_i}}\right],$$
which has degree
$
(n-s)(a_{r_1}+\cdots +a_{r_s})+\sum_{i\in R}b_i.
$
The parts of $D$ contributing to the degree in $x_{r_s}$ are
 $$E_{\mathbf{r},\mathbf{k}}\left[\prod_{j\ne r_0,\dots, r_s}\lrq{x_0}{x_jq^h}{}_{\!\!h}\right],$$
which has degree $(n-s)h$.

Thus the total degree of $\Qrk$ in $x_{r_s}$ is
\begin{align}\label{deg}
d_t=(n-s)(a_{r_1}+\cdots +a_{r_s} - h)+\sum_{i\in R}b_i.
\end{align}

The hypothesis \eqref{deg1} implies that $d_t<0$, so $\Qrk$ is
proper in $x_{r_s}$. Next we apply Lemma \ref{lem-almostprop}. For
any rational function $F$ of $x_{r_s}$ and integers $j$ and $k$, let
$T_{j,k} F$ be the result of replacing $x_{r_s}$ with
$x_{j}q^{k-k_s}$ in $F$. Since $x_{r_s}q^{k_s}/(x_jq^k)$ is small
when $j>r_s$ and is large when $j<r_s$, Lemma \ref{lem-almostprop}
gives
\begin{equation}
\label{e-TQ} \CT_{x_s} \Qrk =\sum_{r_s < r_{s+1}\le n\atop 1\le
   k_{{s+1}}\le h} T_{r_{s+1},k_{s+1}} \left[\Qrk
   \left(1-\frac{x_{r_s}q^{k_s}}{x_{r_{s+1}}q^{k_{s+1}}}\right)\right].
\end{equation}
We must show that the right-hand side of $\eqref{e-TQ}$ is equal to
the right-hand side of \eqref{qh4}. Set $\br'=(r_1,\dots,
r_s, r_{s+1})$ and $\bk'=(k_1,\dots, k_s, k_{s+1})$. Then the
equality follows easily from the identity
\begin{equation}
\label{e-TE} T_{r_{s+1},k_{s+1}}\circ \Erk= E_{\br',\bk'}.
\end{equation}
To see that \eqref{e-TE} holds, we have
$$
(T_{r_{s+1},k_{s+1}}\circ \Erk)\, x_{r_i}
  =T_{r_{s+1},k_{s+1}}\, \left[ x_{r_s}q^{k_s-k_i}\right]
  = x_{r_{s+1}}q^{k_{s+1}-k_i}= E_{\br',\bk'}\,  x_{r_i},
$$
and if $j\notin\{r_0,\dots, r_s\}$ then $(T_{r_{s+1},k_{s+1}}\circ
\Erk)\, x_{j}=x_j=  E_{\br',\bk'}\, x_{j}$.
\end{proof}

\medskip
Now we concentrate on proving our main lemmas. In what follows,
unless specified otherwise, we assume that
$M(\mathbf{x})=x_{j_1}^{p_1}\cdots
x_{j_{\nu}}^{p_{\nu}}\big/(x_{i_1}x_{i_{2}}\cdots x_{i_{m}})$, where the
$j$'s are different from the $i$'s, the $p$'s are positive
integers with $\sum_{i=1}^{\nu}p_i=m$, $n\geq j_{\nu}>\cdots>j_1>0$ and
$n>i_m>\cdots>i_1=0$. Note that the assumptions $i_1=0$ and $i_m<n$
are supported by Lemma \ref{l-pi}.

\begin{lem}\label{lem-lead2}
Let $M(\mathbf{x})$ be as above. If Lemma $\ref{lem-lead1}$
does not apply, then there is a subset $T=\{t_1,t_2,\dots,t_d\}$ of
$I$ such that:  $h=a-\sigma(T)+1$, $Q(h\mid \mathbf{r;k})$ is almost
proper in $x_{n}$, and
$R=\{0,1,\ldots,\widehat{t_1},\ldots,\widehat{t_d},\ldots,n\}$,
where $\widehat{t}$ denotes the omission of $t$.
\end{lem}
\begin{proof}
Since Lemma \ref{lem-lead1} does not apply, we must have
$k_i>a_{r_1}+\cdots+a_{r_s}$ for some $i$ with $1
         \leq i \leq s<n$. It follows that
         $h>a_{r_1}+\cdots+a_{r_s}$.

Let $T=I \setminus R$ denoted by $\{t_1,\ldots,t_d\}$. Then by
\eqref{deg}, the degree in $x_{r_s}$ of $Q(h\mid
\mathbf{r};\mathbf{k})$
         is given by
\begin{align*}
d_t=(n-s)(a_{r_1}+\cdots+a_{r_s}-h)+\sum_{i=1}^{\nu}p_i\chi(j_i\in
R)-(m-d).
\end{align*}
The hypothesis implies that $d_t\ge 0$. This is equivalent to
\begin{align*}
h-(a_{r_1}+\cdots +a_{r_s})\le \frac{\sum_{i=1}^{\nu}p_i\chi(j_i\in
R)-(m-d)}{n-s}.
\end{align*}
Notice that $s\leq n-d$ and $\sum_{i=1}^{\nu}p_i\chi(j_i\in R)\leq m$.
It follows that
$$h-(a_{r_1}+\cdots
+a_{r_s})\le \frac{\sum_{i=1}^{\nu}p_i\chi(j_i\in R)-(m-d)}{n-s}\leq
\frac{m-(m-d)}{n-(n-d)}=1,$$ and the equality holds only when
$s=n-d$ and $\sum_{i=1}^{\nu}p_i\chi(j_i\in R)= m$. The former condition
is sufficient, since if $s=n-d$ then every $j_i$ belongs to $R$.
Thus we can conclude that $h=a_{r_1}+\cdots+a_{r_s}+1$ and $d_t=0$.
This is equivalent to say that $h=a-(a_{t_1}+\cdots+a_{t_d})+1$ and
$Q(h\mid \mathbf{r;k})$ is almost proper in $x_{r_s}$. Since
$i_m<n$, we have $
R=\{0,1,\ldots,\widehat{t_1},\ldots,\widehat{t_d},\ldots,n\}. $
\end{proof}
\begin{proof}[Proof of  Main Lemma {\bf 1}]
By definition \eqref{qh} of $Q(h)$ we see that
$\CT_{\mathbf{x}}Q(-a_{0})$ equals the left-hand side of
\eqref{main-lem} if we take $M(\mathbf{x})=x_{j_1}^{p_1}\cdots
x_{j_{\nu}}^{p_{\nu}}/(x_{i_1}x_{i_2}\cdots x_{i_m})$.

Fix nonnegative integers $a_1,\dots,a_n$. Clearly if $a_0=0$, then
the left-hand side of \eqref{main-lem} is
\begin{align*}
\CT_{\mathbf{x}}\frac{x_{j_1}^{p_1}\cdots x_{j_{\nu}}^{p_{\nu}}}{x_{0}x_{i_2}\cdots x_{i_m}}
\prod_{j=1}^{n} \left(\frac{x_j}{x_0}q\right)_{a_j}
 \prod_{1\leq i<j\leq n} \left(\frac{x_i}{x_j}\right)_{a_i}\left(\frac{x_j}{x_i}q\right)_{a_j}.
\end{align*}
Since the above Laurent polynomial contains only negative powers in
$x_0$, its constant term in $x_0$ equals zero.

Now we prove by induction on $n-s$ that
$$\CT_{\mathbf{x}} \Qrk = 0,\ \ \mbox{if}\ h\in \{1,\ldots,a+1\}\setminus \{a-\sigma(T)+1\mid
T\subseteq I\}.$$
Note that taking constant term with respect to a variable that
does not appear has no effect. Also note that $h\ne
1+a-\sigma(\varnothing)=1+a_1+\cdots+a_n$.

We may assume that $s\le n$ and $0<r_1<\cdots<r_s\le n$, since
otherwise $\Qrk$ is not defined. If $s=n$ then $r_i$ must equal $i$
for $i=1,\dots ,n$. Thus $\Qrk=Q(h\Mid 1,2,\dots, n; k_{1},
k_{2},\dots, k_{n})$, which is 0 by part (i) of Lemma
\ref{lem-lead1} and the fact that $k_i\le h\le a_1+\cdots + a_n$ for
each $i$.

Now suppose $0\le s<n$. Since $b_0=-1$, the condition
$h>\frac{b_0}{n}=-\frac{1}{n}$ always holds. If part (i) of Lemma
\ref{lem-lead1} applies, then $\Qrk=0$. Otherwise, by Lemma
\ref{lem-lead2}, part (ii) of Lemma \ref{lem-lead1} applies and
\eqref{qh4} holds. Therefore, applying $\CT_{\mathbf{x}}$ to both
sides of \eqref{qh4} gives
$$
\CT_{\mathbf{x}}\Qrk =\sum_{r_s < r_{s+1}\le n\atop 1\le
k_{{s+1}}\le h}
    \CT_{\mathbf{x}}Q(h\Mid r_1,\dots, r_s, r_{s+1};k_1,\dots,
    k_s,k_{s+1}).
$$
By induction, every term on the right is zero.
\end{proof}

\section{Proof of Main Lemma 2}

The proof of Main Lemma 2 relies on Lemma \ref{lem-almostprop} for
almost proper rational functions. It involves complicated
computations. By the proof of Main Lemma 1, Lemma \ref{lem-lead2}
describes all cases for $\CT_\mathbf{x} Q(h\mid
\mathbf{r},\mathbf{k})\ne 0$. To evaluate such cases, we need the
following two lemmas.

\begin{lem}\label{lem-u0}
\begin{align}\label{e-u0}
&\prod_{l=1}^n\frac{(q^{-\sum_{i=l}^nw_i})_{w_l}}{(q)_{\sum_{i=l}^nw_i}(q^{-\sum_{i=1}^{l-1}w_i})_{\sum_{i=1}^{l-1}w_i}}
\prod_{1\leq i<j\leq n}\big(q^{-\sum_{l=i}^{j-1}w_l}\big)_{w_i}\big(q^{\sum_{l=i}^{j-1}w_l+1}\big)_{w_j}\nonumber \\
=&(-1)^{w}q^{-{w+1\choose 2}}\frac{(q)_{w}}{(q)_{w_1}\cdots (q)_{w_{n}}},
\end{align}
where $w=w_1+\cdots+w_{n}$.
\end{lem}
\begin{proof}
Denote the left-hand side of \eqref{e-u0} by $H_n$ and the
right-hand side by $G_n$. Clearly we have $H_1=G_1$. To show that
$H_n=G_n$, it suffices to show that $H_n/H_{n-1}=G_n/G_{n-1}$ for
$n\ge 2$. We have

\begin{align*}
\frac{H_{n}}{H_{n-1}}
=&\frac{(q^{-w_{n}})_{w_n}}{(q)_{w_n}(q^{-w_1-\cdots-w_{n-1}})_{w_1+\cdots+w_{n-1}}}
\prod_{l=1}^{n-1}\frac{(q^{-w_l-\cdots-w_n})_{w_l}}{(q^{w_l+\cdots+w_{n-1}+1})_{w_n}
(q^{-w_l-\cdots-w_{n-1}})_{w_l}}\nonumber \\
&\cdot\prod_{l=1}^{n-1}(q^{-w_l-\cdots-w_{n-1}})_{w_l}(q^{w_l+\cdots+w_{n-1}+1})_{w_n}\nonumber\\
=&\frac{(-1)^{w_n}q^{-{w_n+1\choose 2}}}{(-1)^{w-w_n}q^{-{w-w_n+1\choose 2}}(q)_{w-w_n}}
\prod_{l=1}^{n-1}(-1)^{w_l}q^{-{w_l+1\choose 2}-w_l(w_{l+1}+\cdots+w_n)}(q^{w_{l+1}+\cdots+w_n+1})_{w_l}.
\end{align*}
Since it is straightforward to show that
$$
\prod_{l=1}^{n-1}q^{-{w_l+1\choose
2}-w_l(w_{l+1}+\cdots+w_n)}=q^{-{w-w_n+1\choose 2}-w_n(w-w_n)}$$ and
that
$$\prod_{l=1}^{n-1}(q^{w_{l+1}+\cdots+w_n+1})_{w_l}=(q^{w_n+1})_{w-w_n},
$$
we have
\begin{align*}
\frac{H_{n}}{H_{n-1}}=(-1)^{w_n}q^{-{w_n+1\choose 2}-w_n(w-w_n)}\frac{(q^{w_n+1})_{w-w_n}}{(q)_{w-w_n}},
\end{align*}
which is equal to $G_n/G_{n-1}$.
\end{proof}

\medskip
For fixed subset $T=\{t_1,t_2,\dots,t_d\}$ of $I$, we let
$h^{*}=a-\sigma(T)+1=w+1$, $\mathbf{r}^{*}=
(1,\ldots,\widehat{t_1},\ldots,\widehat{t_d},\ldots,n)$, and
$\mathbf{k^{*}}=(k_1,\ldots,k_{n-d})$ with
$k_{l}=\sum_{i=r_l}^nw_i+1$. Let
\begin{align}\label{defNl}
N_{l}=\# \{t_j<l\mid t_j\in T\},
\end{align}
where $\# S$ is the cardinality of the set $S$. Then
$E_{\mathbf{r^*},\mathbf{k^*}}x_{i}$ is
$x_{n}q^{k_{n-d}-k_{i-N_i}}$ for $i\notin T$, and is $x_i$ for $i\in T$.
For $i\notin T$, we have $k_{n-d}-k_{i-N_i}=w_{n}-\sum_{l=i}^nw_l$.

\begin{lem}\label{lem-compute}
Let $T$ be a subset of $I$. Then
\begin{align}\label{e-u1}
\CT_{\mathbf{x}}Q(h^*\mid \mathbf{r^{*}}; \mathbf{k^{*}})=(-1)^{w+d}q^{L^*(T)}\frac{(q)_w(q)_{a-w}}{(q)_{a_1}\cdots
(q)_{a_n}},
\end{align}
where  {\small$$L^*(T)=\sum_{l\in
I}\sum_{i=l}^nw_i-\sum_{l=1}^{\nu}p_l\sum_{i=j_l}^nw_i-{w+1\choose
2}-1.$$}
\end{lem}

\begin{proof}

By Lemma \ref{lem-lead2}, $Q(h^*\mid \mathbf{r^{*}};
\mathbf{k^{*}})$ is almost proper in $x_{n}$. Let $R^{*}=
\{r_1,\dots,r_s\}=\{1,\ldots,n\}\setminus T$, $s=n-d$.

It is straightforward to check that for any $1\leq i<j\leq n$
\begin{equation}\label{e-LC1}
\LC_{x_{n}} E_{\mathbf{r^*},\mathbf{k^*}} \Big(\frac{x_i}{x_j}\Big)_{\!a_i}\Big(\frac{x_j}{x_i}q\Big)_{\!a_j}=
\left\{ \begin{aligned}
  & \Big(-\frac{1}{x_i}\Big)^{\!w_j}q^{{w_j+1\choose 2}+(w_{n}-\sum_{l=j}^nw_l)w_j},  &\text{ if }&   i\notin R^*, j\in R^*,  \\
  & \Big(-\frac{1}{x_j}\Big)^{\!w_i}q^{{w_i\choose 2}+(w_{n}-\sum_{l=i}^nw_l)w_i},  &\text{ if }&   i\in R^*, j\notin R^*,\\
  & \big(q^{-\sum_{l=i}^{j-1}w_l}\big)_{w_i}\big(q^{\sum_{l=i}^{j-1}w_l+1}\big)_{w_j},  &\text{ if }&   i,j \in R^*, \\
  & \Big(\frac{x_i}{x_j}\Big)_{\!a_i}\Big(\frac{x_j}{x_i}q\Big)_{\!a_j}, &\text{ if }&   i,j\notin R^*.
                          \end{aligned} \right.
                          \end{equation}
For convenience, we always assume $i<j$ within this proof if $i$ and
$j$ appears simultaneously.

Recall that $M(\mathbf{x})=x_{j_1}^{p_1}\cdots
x_{j_{\nu}}^{p_{\nu}}/(x_{i_1}x_{i_2}\cdots x_{i_m})$, we have
\begin{align}\label{e-LC2}
E_{\mathbf{r^*},\mathbf{k^*}}M(\mathbf{x})
=\frac{x_{n}^mq^{\sum_{i=1}^{\nu}p_i(k_{n-d}-k_{j_i-N_{j_i}})}}
{x_{n}^{m-d}q^{(m-d)k_{n-d}-\sum_{l\in I\setminus
T}k_{l-N_l}}x_{t_1}\cdots x_{t_d}} =
\frac{x_{n}^dq^{L_1(d)}}{x_{t_1}\cdots x_{t_d}},
\end{align}
where
\begin{align}\label{e-L1}
L_1(d)=dw_{n}+\sum_{l\in I\setminus
T}\sum_{i=l}^nw_i-\sum_{i=1}^{\nu}p_i\sum_{l=j_i}^nw_l-1.
\end{align}

It is easy to see that
\begin{align}\label{e-LC3}
\LC_{x_n}E_{\mathbf{r^*},\mathbf{k^*}}\prod_{l=1}^n\Big(\frac{x_l}{x_0}q\Big)_{a_l}=
\prod_{l\in R^*}(q^{-\sum_{i=l}^nw_i})_{w_l},
\end{align}
and that{\small\begin{align}\label{e-LC4}
E_{\mathbf{r^*},\mathbf{k^*}}\frac{\prod_{i=1}^{n-d}\left(1-
x_0/(x_{r_i}q^{k_i})\right)}{\prod_{l=1}^n\big(x_0/(x_lq^{h^*})\big)_{h^*}}
=\frac{1}{\prod_{l\in
R^*}(q)_{\sum_{i=l}^nw_i}(q^{-\sum_{i=1}^{l-1}w_i})_{\sum_{i=1}^{l-1}w_i}
\prod_{l\notin R^*}(x_nq^{w_n-w}/x_l)_{w+1}}.
\end{align}}

By the definition of $Q(h)$ in \eqref{qh1}, we have
\begin{align}
Q&(h^*\mid \mathbf{r^{*}}; \mathbf{k^{*}})\nonumber \\
&=E_{\mathbf{r^*},\mathbf{k^*}}M(x)\prod_{j=1}^{n} \frac{(x_jq/x_0)_{a_j}}
{\big(x_0/(x_jq^{h^*})\big)_{h^*}}
 \prod_{1\leq i<j\leq n} \left(\frac{x_i}{x_j}\right)_{\!a_i}\left(\frac{x_j}{x_i}q\right)_{\!a_j}
 \prod_{i=1}^{n-d}\left(1-
x_0/(x_{r_i}q^{k_i})\right).
\end{align}

 Apply Lemma
\ref{lem-almostprop} with respect to $x_n$.  Since $Q(h^*\mid
\mathbf{r^{*}}; \mathbf{k^{*}})$ has no small factors in the
denominator, the summation part in \eqref{e-almostprop} equals 0.
Thus the result can be written as
$$\LC_{x_n} E_{\mathbf{r^*},\mathbf{k^*}}M(x)\prod_{j=1}^{n} \frac{(x_jq/x_0)_{a_j}}
{\big(x_0/(x_jq^{h^*})\big)_{h^*}}
 \prod_{1\leq i<j\leq n} \left(\frac{x_i}{x_j}\right)_{\!a_i}\left(\frac{x_j}{x_i}q\right)_{\!a_j}
 \prod_{i=1}^{n-d}\left(1-
x_0/(x_{r_i}q^{k_i})\right).$$ Substituting \eqref{e-LC1},
\eqref{e-LC2}, \eqref{e-LC3}, and \eqref{e-LC4} into the result, and
then collecting similar terms, we can write
\begin{align}\label{e-w0}
\CT_{\mathbf{x}}Q(h^*\mid \mathbf{r^{*}}; \mathbf{k^{*}})
=q^{L_1(d)}A_1A_2\CT_\mathbf{x} B_1B_2.
\end{align}
Here $q^{L_1(d)}A_1$ is the collection of all powers in $q$ (only from
(\ref{e-LC4}, \ref{e-LC1})) given by
 {\small
\begin{align*}
A_1=&\prod_{l\notin R^*}q^{{w+1\choose
2}-(w+1)w_n} \prod_{i\notin R^*,j\in R^*}q^{{w_j+1\choose
2}+(w_n-\sum_{l=j}^nw_l)w_j}\prod_{i\in R^*, j\notin
R^*}q^{{w_i\choose 2}+(w_{n}-\sum_{l=i}^nw_l)w_i};
\end{align*}}
$A_2$ is the collection of all $q$-factorials (only from (\ref{e-LC3},
\ref{e-LC4}, \ref{e-LC1})) given by {\small
\begin{align*}
A_2=\prod_{l\in
R^*}\frac{(q^{-\sum_{i=l}^nw_i})_{w_l}}{(q)_{\sum_{i=l}^nw_i}(q^{-\sum_{i=1}^{l-1}w_i})_{\sum_{i=1}^{l-1}w_i}}
\prod_{i,j\in
R^*}\big(q^{-\sum_{l=i}^{j-1}w_l}\big)_{w_i}\big(q^{\sum_{l=i}^{j-1}w_l+1}\big)_{w_j};
\end{align*}}
$B_1$ is the collection of all monomial factors (only from
(\ref{e-LC2}, \ref{e-LC4}, \ref{e-LC1})) given by
 {\small
\begin{align}\label{e-B1} B_1=&\frac{1}{x_{t_1}\cdots
x_{t_d}}\prod_{l\notin R^*}(-1)^{w+1}x_{l}^{w+1} \prod_{i\notin
R^*,j\in R^*}\big(-1/x_{i}\big)^{w_j}\prod_{i\in R^*, j\notin
R^*}\big(-1/x_j\big)^{w_i}=(-1)^d;
\end{align}}
and $B_2$ is the collection of all $q$-factorials containing
variables (only from (\ref{e-LC1})) given by {\small
\begin{align*}
B_2=\prod_{i,j\notin R^*}(x_i/x_j)_{a_i}(x_jq/x_i)_{a_j}=
D_d(x_{t_1},\dots,x_{t_d};a_{t_1},\dots,a_{t_d};q).
\end{align*}}
(Note that for the $q$-Dyson Theorem, $M(x)=1$, $T=I=\varnothing$,
and hence $B_2=1$, so we do not need the next paragraph for our
alternative proof of Theorem \ref{thm-dyson}.)

It follows by Theorem \ref{thm-dyson} and \eqref{e-B1} that
\begin{align}\label{e-w3}
\CT_{\mathbf{x}} B_1B_2=\CT_{\mathbf{x}}(-1)^d\prod_{i,j\notin
R^*}(x_i/x_j)_{a_i}(x_j/x_iq)_{a_j}=(-1)^d\frac{(q)_{a-w}}{\prod_{l\in
T}(q)_{a_l}}.
\end{align}

Recall that $w_i=0$ if $i\not\in R^*$. By Lemma \ref{lem-u0} we have
\begin{align}\label{e-w1}
A_2=(-1)^{w}q^{-{w+1\choose 2}}\frac{(q)_{w}}{\prod_{l\in
R^*}(q)_{w_l}}.
\end{align}

Let
$A_1=q^{L_2(d)}$,
where
{\small\begin{align*}
L_2(d)=&\sum_{l\notin R^*}\left[{w+1\choose 2}-(w+1)w_n\right]+\sum_{i\notin R^*,j\in R^*}
\left[{w_j+1\choose 2}+\Big(w_n-\sum_{l=j}^nw_l\Big)w_j \right] \nonumber \\
&+\sum_{i\in R^*,j\notin R^*}\left[{w_i+1\choose 2}+\Big(w_n-\sum_{l=i}^nw_l\Big)w_i-w_i\right].
\end{align*}}
We claim that
\begin{equation}\label{e-L-2d}
L_2(d)=\widetilde{L}_2(d)=-dw_n+dw-\sum_{l\in T}\sum_{k=1}^{l-1}w_k.
\end{equation}

It is clear that $L_2(0)=\widetilde{L}_2(0)=0$. Therefore to show
that $L_2(d)=\widetilde{L}_2(d)$ it suffices to show that
$L_2(d)-L_2(d-1)=\widetilde{L}_2(d)-\widetilde{L}_2(d-1)$ for $d\geq
1$.

Since $w_i=0$ for $i\in T$, we have
{\small\begin{align*}
L_2(d)-L_2(d-1)
=&{w+1\choose 2}-(w+1)w_n+\sum_{j=t_d+1}^n
\left[{w_j+1\choose 2}+\Big(w_n-\sum_{l=j}^nw_l\Big)w_j \right] \nonumber \\
&\quad+\sum_{i=1}^{t_d-1}\left[{w_i+1\choose 2}+\Big(w_n-\sum_{l=i}^nw_l\Big)w_i-w_i\right] \nonumber \\
=&{w+1\choose 2}-(w+1)w_n+ \sum_{j=1}^n\left[{w_j+1\choose
2}+\Big(w_n-\sum_{l=j}^nw_l\Big)w_j \right]-\sum_{i=1}^{t_d-1}w_i.
\end{align*}}

Simplifying the above equation, we obtain {\small\begin{align*}
L_2(d)-L_2(d-1) =&{w+1\choose 2}-(w+1)w_n+
\sum_{j=1}^n{w_j+1\choose 2}+w_nw-\sum_{j=1}^n\sum_{l=j}^nw_lw_j-\sum_{i=1}^{t_d-1}w_i \nonumber \\
=&{w+1\choose 2}-w_n+
\sum_{j=1}^n{w_j+1\choose 2}-\sum_{i<j}w_iw_j-\sum_{i=1}^nw_i^2-\sum_{i=1}^{t_d-1}w_i.
\end{align*}}
Using the fact ${w+1\choose 2}=\sum_{i=1}^n{w_i+1\choose
2}+\sum_{i<j}w_iw_j$, we get
\begin{align*}
L_2(d)-L_2(d-1)
=&-w_n+2\sum_{j=1}^n{w_j+1\choose 2}-\sum_{i=1}^nw_i^2-\sum_{i=1}^{t_d-1}w_i \nonumber \\
=&-w_n+w-\sum_{i=1}^{t_d-1}w_i,
\end{align*}
which equals $\widetilde{L}_2(d)-\widetilde{L}_2(d-1)$. Thus the
claim follows.

Substituting \eqref{e-w3}, \eqref{e-w1}, and $A_1=q^{L_2(d)}$ (with
\eqref{e-L-2d}) into \eqref{e-w0} and simplifying yields
\begin{align*}
\CT&_{\mathbf{x}}Q(h^*\mid \mathbf{r^{*}};
\mathbf{k^{*}})=(-1)^{d+w}q^{L_1(d)+L_2(d)-{w+1\choose
2}}\frac{(q)_{w}(q)_{a-w}}{(q)_{a_1}\cdots (q)_{a_n}}.
\end{align*}

Therefore
{\small\begin{align*} L^*(T)=&L_1(d)+L_2(d)-{w+1\choose 2}\nonumber \\
=&dw_{n}+\sum_{l\in I\setminus
T}\sum_{i=l}^nw_i-\sum_{i=1}^{\nu}p_i\sum_{l=j_i}^nw_l-1
-dw_n+dw-\sum_{l\in T}\sum_{k=1}^{l-1}w_k-{w+1\choose 2}\nonumber \\
=&\sum_{l\in I\setminus
T}\sum_{i=l}^nw_i-\sum_{i=1}^{\nu}p_i\sum_{l=j_i}^nw_l-1 +dw-\sum_{l\in
T}\sum_{k=1}^{l-1}w_k-{w+1\choose 2}.
\end{align*}}
Since $dw$ can be written as $\sum_{l\in T}\sum_{k=1}^nw_k$, we have
{\small\begin{align*} L^*(T)=&\sum_{l\in I\setminus
T}\sum_{i=l}^nw_i-\sum_{i=1}^{\nu}p_i\sum_{l=j_i}^nw_l-1
+\sum_{l\in T}\sum_{k=l}^{n}w_k-{w+1\choose 2}\nonumber \\
=&\sum_{l\in I}\sum_{i=l}^nw_i-\sum_{i=1}^{\nu}p_i\sum_{l=j_i}^nw_l-1
-{w+1\choose 2}.
\end{align*}}
\end{proof}

\begin{proof}[Proof of  Main Lemma {\bf 2}]
 Applying Lemma
\ref{lem-almostprop} gives \eqref{qh2} as follows.
\begin{align*}
\CT_{x_0}Q(h)=\sum_{\substack{0<r_1\leq n,\\ 1\leq k_1\leq
h}}Q(h\mid r_1;k_1).
\end{align*}
Iteratively apply Lemma \ref{lem-lead1} to each summand when
applicable. In each step, we need to deal with a sum of terms like
$Q(h\mid r_1,\dots,r_s;k_1,\dots,k_s)$. For such summand, we apply
Lemma \ref{lem-lead1} with respect to $x_{r_s}$. The summand is
taken to $0$ if part (i) applies, and is taken to a sum if part (ii)
applies. In the latter case, the number of variables decreases by
one. Since there are only $n+1$ variables, the iteration terminates.
Note that if $r_s=n$ and part (ii) applies, the summand will be
taken to $0$. So finally we can write
\begin{align*}
\CT_{\mathbf{x}}Q(h)=\CT_{\mathbf{x}}\sum_{r_1,\dots,r_s,
k_1,\dots,k_s}Q(h\mid r_1,\dots,r_s;k_1,\dots,k_s),
\end{align*}
where the sum ranges over all $r$'s and $k$'s with $0<r_1<\cdots
<r_s\leq n, 1\leq k_1,k_2,\dots,k_s\leq h$ such that Lemma
\ref{lem-lead1} does not apply. Note that we may have different $s$.

By Lemma \ref{lem-lead2}, Lemma \ref{lem-lead1} does not apply only
if there is a subset $T=\{t_1,\dots,t_d\}$ of $I$ such that
$(r_1,\dots,r_s)=(1,\ldots,\widehat{t_1},\ldots,\widehat{t_d},\ldots,n)$,
and $h=a-\sigma(T)+1$. 
So the sum becomes
\begin{align*}
\CT_{\mathbf{x}}Q(h)= &\CT_{\mathbf{x}}\sum_{T} \sum_{1\leq
k_1,\ldots,k_{n-d}\leq h}Q(h\mid \mathbf{r^{*}};\mathbf{k}),
\end{align*}
where $T$ ranges over all $T\subseteq I$ such that
$a-\sigma(T)+1=h$.

For each fixed subset $T$ of $I$ as above, we show that almost every
$Q(h\mid \mathbf{r^{*}};\mathbf{k})$ vanishes. Notice that
$E_{\mathbf{r^{*}},\bk}x_i=x_n^{k_{n-d}-k_{i-N_i}}$ for $i\notin T$
with $N_i$ defined as in \eqref{defNl}. Rename the parameters $a_i$
by $w_i$ for $i\not\in T$, and set $w_i=0$ for $i\in T$. The
expression becomes easy to describe.

If $1\le k_{i-N_i}\le w_i$ for some $i\notin T$, then $Q(h\mid \mathbf{r^{*}};\mathbf{k})$ has the factor
{\small$$E_{\mathbf{r^{*}},\bk}
\left[\lrq{x_{i}}{x_{0}}{q}_{\!\!a_{i}} \right]
=\lrq{x_{n}q^{k_{n-d}-k_{i-N_i}}}{x_{n}q^{k_{n-d}}}{q}_{\!\!w_i}
=(q^{1-k_{i-N_i}})_{w_i}=0.$$}

If $-w_{j}\le k_{i-N_i}-k_{j-N_j}\le w_{i}-1$,
where $i<j$ and $i,j\notin T$, then $Q(h\mid \mathbf{r^{*}};\mathbf{k})$ has
the factor
$$E_{\mathbf{r^{*}},\bk}\, \left[\lrq{x_{i}}{x_{j}}{}_{\!\!a_{i}}\!\!\lrq{x_{j}}{x_{i}}{q}_{\!\!a_{j}}\right]
=E_{\mathbf{r^{*}},\bk}\, \left[\lrq{x_{i}}{x_{j}}
{}_{\!\!w_{i}}\!\!\lrq{x_{j}}{x_{i}}{q}_{\!\!w_{j}}\right], $$ which
is equal to {\small\begin{align*}
E_{\mathbf{r^{*}},\bk}\left[q^{\binom{w_{j}+1}{2}}
\left(-\frac{x_{j}}{x_{i}}\right)^{w_{j}}\!
\lrq{x_{i}}{x_{j}}{q^{-w_{j}}}_{\!\!w_{i}+w_{j}} \right]
=q^{\binom{w_{j}+1}{2}}
(-q^{k_{i-N_i}-k_{j-N_j}})^{w_{j}}(q^{k_{j-N_j}-k_{i-N_i}-w_{j}})_{w_{i}+w_{j}}=0.
\end{align*}}
If neither of the above two cases happen, then by Lemma
\ref{lem-import} for the case $s=n-d$, we see that $\mathbf{k}$ must
equal $\mathbf{k^{*}}$ given by
{\small\begin{equation*}
\mathbf{k^{*}}=\left(\sum_{i=r_1}^nw_i+1,
\sum_{i=r_2}^nw_i+1,\,\ldots,\sum_{i=r_{n-d}}^nw_i+1\right).
\end{equation*}}
Therefore, for every $T$, all $Q(h\mid \mathbf{r^{*}};\mathbf{k})$
vanish except for $Q(h\mid \mathbf{r^{*}};\mathbf{k^*})$. It follows
that
 $$\CT_{\mathbf{x}}Q(h^*)=\CT_{\mathbf{x}}\sum_{T}Q(h^*\mid \mathbf{r^{*}}; \mathbf{k^{*}})
 =\sum_{T}\CT_{\mathbf{x}}Q(h^*\mid \mathbf{r^{*}}; \mathbf{k^{*}}).$$
Thus the proof is completed by Lemma \ref{lem-compute}.
\end{proof}

\section{Concluding Remark}

For the equal parameter case, Stembridge \cite{stembridge1987}
studied the constant terms for general monomials $M(\mathbf{x})$ and
obtained recurrence formulas. However, explicit formulas are
obtained only for $M(\mathbf{x})=x_{j_1}^{p_1}\cdots
x_{j_{\nu}}^{p_{\nu}}\big/(x_{i_1}x_{i_{2}}\cdots x_{i_{m}})$, just as we
discussed. These formulas are called first layer formulas. For the unequal parameter case,
 our method may be used to evaluate the
constant terms for monomials like $M(\mathbf{x})=x_sx_t/x_0^2$, but
the explicit formula will be too complicated. We can expect that
other types of $q$-Dyson style constant terms can be solved in a
similar way.

\vspace{.2cm} \noindent{\bf Acknowledgments.} Lun Lv and Yue Zhou
would like to acknowledge the helpful guidance of their supervisor
William Y.C. Chen. This work was supported by the 973 Project, the
PCSIRT project of the Ministry of Education, the Ministry of Science
and Technology and the National Science Foundation of China.

\end{document}